\documentclass[runningheads]{llncs}
\usepackage{graphicx}
\usepackage{url}

\usepackage{hyperref}

\usepackage{amssymb,amsmath, amsrefs, amsfonts}
 
\newcommand{\Z}{{\textsf{\textup{Z}}}}

\newtheorem{thm}{Theorem}

\newtheorem{cor}[thm]{Corollary}
\newtheorem{defi}[thm]{Definition}
\newtheorem{rem}[thm]{Remark}
\newtheorem{nota}[thm]{Notation}

\newtheorem{princ}[thm]{Principle}

\newtheorem{ack}[thm]{Acknowledgement}

\newcommand\be{\begin{equation}}
\newcommand\ee{\end{equation}} 

\usepackage{amsmath,amsfonts}

\def\bdefi{\begin{defi}\rm}
\def\edefi{\end{defi}}
\def\bnota{\begin{nota}\rm}
\def\enota{\end{nota}}

\def\FIVE{\Pi_{1}^{1}\text{-\textup{\textsf{CA}}}_{0}}

\def\SIXK{\Pi_{k}^{1}\text{-\textsf{\textup{CA}}}_{0}^{\omega}}

\def\ZFC{\textup{\textsf{ZFC}}}
\def\ZF{\textup{\textsf{ZF}}}

\def\L{\textsf{\textup{L}}}

\def\({\textup{(}}
\def\){\textup{)}}

\def\RCAo{\textup{\textsf{RCA}}_{0}^{\omega}}
\def\ACAo{\textup{\textsf{ACA}}_{0}^{\omega}}

\def\bye{\end{document}}
\def\N{{\mathbb  N}}
\def\Q{{\mathbb  Q}}
\def\R{{\mathbb  R}}

\def\SS{\textup{\textsf{S}}}

\def\di{\rightarrow}

\def\asa{\leftrightarrow}

\def\ACA{\textup{\textsf{ACA}}}

\def\INDY{\textup{\textsf{IND}}_{0}}

\def\cocode{\textup{\textsf{cocode}}}

\def\NIN{\textup{\textsf{NIN}}}

\def\NBI{\textup{\textsf{NBI}}}

\def\eps{\varepsilon}

\usepackage{comment}

\begin{document}
\title{On the computational properties of the uncountability of the real numbers\thanks{This research was supported by the \emph{Deutsche Forschungsgemeinschaft} (DFG) via the grant \emph{Reverse Mathematics beyond the G\"odel hierarchy} (SA3418/1-1).}}
%
%
\author{Sam Sanders\inst{1}}
\authorrunning{S.\ Sanders}
%
\institute{Department of Philosophy II, RUB Bochum, Germany \\
\email{sasander@me.com} \\
\url{https://sasander.wixsite.com/academic}}

\setcounter{secnumdepth}{3}
\setcounter{tocdepth}{3}

%
\maketitle              
\begin{abstract}
The uncountability of the real numbers is one of their most basic properties, known (far) outside of mathematics. 
Cantor's 1874 proof of the uncountability of the real numbers even appears in the very first paper on set theory, i.e.\ a historical milestone. Despite this famous status and history, the computational properties of the uncountability of the real numbers have not been studied much. In this paper, we study the following computational operation that witnesses that the real numbers not countable:
\[
\text{\emph{on input a countable set of reals, output a real not in that set.}}
\]
In particular, we formulate a considerable number of operations that are computationally equivalent to the centred operation, working in Kleene's higher-order computability theory based on his S1-S9 computation schemes. Perhaps surprisingly, our equivalent operations involve most basic properties of the \emph{Riemann integral} and Volterra's early work circa 1881. 
\end{abstract}
\section{Introduction}\label{intro}
\subsection{Motivation and overview}
Like Hilbert (\cite{hilbertendlich}), we believe the infinite to be a central object of study in mathematics. 
That the infinite comes in `different sizes' is a relatively new insight, due to Cantor around 1874 (\cite{cantor1}), in the guise of the \emph{uncountability of $\R$}, also known simply as \emph{Cantor's theorem}. 
Cantor's paper \cite{cantor1} is heralded as the first paper in set theory, complete with its own Wikipedia page \cite{wica}.

\smallskip

Despite the aforementioned famous status, the computational properties of the uncountability of $\R$ have not been studied much.  
In this paper, we study the following computational operation witnessing that $\R$ is not countable:
\begin{center}
\text{\emph{on input a countable set $A\subset \mathbb{R}$, output $y\in \mathbb{R}\setminus A$.}}
\end{center}
In particular, we show that this operation is computationally equivalent to the following ones.  Note that a \emph{regulated} function has left and right limits everywhere, as studied by Bourbaki for Riemann integration (see Section \ref{cdef}).
\begin{enumerate}
\renewcommand{\theenumi}{(\roman{enumi})}
\item On input regulated $f:[0,1]\di \R$, output a point $x\in [0,1]$ where $f$ is continuous (or quasi-continuous or lower semi-continuous or Darboux).
\item On input regulated $f:[0,1]\di [0,1]$ with Riemann integral $\int_{0}^{1}f(x)dx=0$, output $x\in [0,1]$ with $f(x)=0$ (Bourbaki, \cite{boereng}*{p.\ 61, Cor.\ 1}).
\item (Volterra) On input regulated $f,g:[0,1]\di \R$, output a real $x\in [0,1]$ such that $f$ and $g$ are both continuous or both discontinuous at $x$. \label{volkert1}
\item (Volterra) On input regulated $f:[0,1]\di \R$, output either $q\in \Q\cap [0,1]$ where $f$ is discontinuous, or output $x\in [0,1]\setminus \Q$ where $f$ is continuous.\label{volkert2}
\item On input regulated $f:[0,1]\di \R$, output $y\in (0,1)$ where the function $F(x):=\lambda x.\int_{0}^{x}f(t)dt$ is differentiable with derivative equal to $f(y)$.
\item On input regulated $f:[0,1]\di \R$, output $a, b\in  [0,1]$ such that the set $\{ x\in [0,1]:f(a)\leq f(x)\leq f(b)\}$ is infinite.\
\end{enumerate}
A full list may be found in Theorem \ref{flunk}.  
We work in Kleene's higher-order computability theory, based on his S1-S9 computation schemes (see Section \ref{prelim}).  
Like in \cite{dagsamXIII}, a much weaker notion of computability should suffice, namely based on a fragment of G\"odel's $T$.
We note that by \cite{dagsamXII}*{Theorem 18}, the operations in the above list are \emph{hard\footnote{The functional $\SS_{k}^{2}$ from Section \ref{lll} can decide $\Pi_{k}^{1}$-formulas, but the centred operation is not computable in $\SS_{k}^{2}$ (or their union).  Kleene's $\exists^{3}$ from Section \ref{lll} computes the centred operation, but the former also yields full second-order arithmetic.\label{klank}} to compute} relative to the usual hierarchy based on comprehension.  
An explanation for this phenomenon is also in Section \ref{prelim}.

\smallskip

Some of the above operations, including items \ref{volkert1} and \ref{volkert2} in the list, stem from Volterra's early work (1881) in the spirit of -but predating- the Baire category theorem, as discussed in Section \ref{vintro}.  

\smallskip

It should not be a surprise that the above equivalences crucially hinge on our definition of `countable set', as discussed in detail in Remark \ref{diunk}.  In a nutshell, the set theoretic definition of `countable set' (based on injections or bijections to $\N$) is not suitable for studying regulated functions; our alternative definition (based on unions over $\N$ of finite sets) is closer to the way countable sets occur `in the wild', at least in the context of regulated functions. 

\smallskip

Finally, this paper is a spin-off from my project with Dag Normann (University of Oslo) on the logical and computational properties of the uncountable.
The interested reader may consult \cite{dagsamIII,dagsamX, dagsamXIII} for further details on our project.

\subsection{Volterra's early work and related results}\label{vintro}
We introduce Volterra's early work from \cite{volaarde2} as it pertains to this paper. 

\smallskip

First of all, the Riemann integral was groundbreaking for a number of reasons, including its ability to integrate functions with infinitely many points of discontinuity, as shown by Riemann himself (\cite{riehabi}). 
A natural question is then `how discontinuous' a Riemann integrable function can be.  In this context, Thomae introduced $T:\R\di\R$ around 1875 in \cite{thomeke}*{p.\ 14, \S20}):
\be\label{thomae}
T(x):=
\begin{cases} 
0 & \textup{if } x\in \R\setminus\Q\\
\frac{1}{q} & \textup{if $x=\frac{p}{q}$ and $p, q$ are co-prime} 
\end{cases}.
\ee
Thomae's function $T$ is integrable on any interval, but has a dense set of points of discontinuity, namely $\Q$, and a dense set of points of continuity, namely $\R\setminus \Q$. 

\smallskip

The perceptive student, upon seeing Thomae's function as in \eqref{thomae}, will ask for a function continuous at each rational point and discontinuous at each irrational one.
Such a function cannot exist, as is generally proved using the Baire category theorem.  
However, Volterra in \cite{volaarde2} already established this negative result about twenty years before the publication of the Baire category theorem.

\smallskip

Secondly, as to the content of Volterra's paper \cite{volaarde2}, we find the following theorem on the first page, where a function is \emph{pointwise discontinuous} if it has a dense set of continuity points.
\begin{thm}[Volterra, 1881]\label{VOL}
There do not exist pointwise discontinuous functions defined on an interval for which the continuity points of one are the discontinuity points of the other, and vice versa.
\end{thm}
Volterra then states two corollaries, of which the following is perhaps well-known in `popular mathematics' and constitutes the aforementioned negative result. 
\begin{cor}[Volterra, 1881]\label{VOLcor}
There is no $\R\di\R$ function that is continuous on $\Q$ and discontinuous on $\R\setminus\Q$. 
\end{cor}
Thirdly, we shall study Volterra's theorem and corollary restricted to regulated functions (see Section \ref{cdef}). 
The latter kind of functions are automatically `pointwise discontinuous' in the sense of Volterra.

\smallskip

Fourth,  Volterra's results from \cite{volaarde2} are generalised in \cite{volterraplus,gaud}.  
The following theorem is immediate from these generalisations. 
\begin{thm}\label{dorki}
For any countable dense set $D\subset [0,1]$ and $f:[0,1]\di \R$, either $f$ is discontinuous at some point in $D$ or continuous at some point in $\R\setminus D$. 
\end{thm}
%
%
In conclusion, the perceptive reader has already noted that most equivalent principles have a `Baire category theorem' flavour to them.  
In particular, the intuitive notion of `small or negligible set' can be formalised in at least two (fairly independent) ways, namely based on `measure' and `density/meagre'.  
The equivalences in this paper do not (and we believe cannot) involve the former, but heavily depend on the latter. 

\subsection{Preliminaries and definitions}
We briefly introduce Kleene's \emph{higher-order computability theory} in Section \ref{prelim}.
We introduce some essential axioms (Section \ref{lll}) and definitions (Section~\ref{cdef}).  A full introduction may be found in e.g.\ \cite{dagsamX}*{\S2}.
Since Kleene's computability theory borrows heavily from type theory, we shall often use common notations from the latter; for instance, the natural numbers are type $0$ objects, denoted $n^{0}$ or $n\in \N$.  
Similarly, elements of Baire space are type $1$ objects, denoted $f\in \N^{\N}$ or $f^{1}$.  An overview of this kind of notations is in Section \ref{appendisch}. 

\subsubsection{Kleene's computability theory}\label{prelim}
Our main results are in computability theory and we make our notion of `computability' precise as follows.  
\begin{enumerate}
\item[(I)] We adopt $\ZFC$, i.e.\ Zermelo-Fraenkel set theory with the Axiom of Choice, as the official metatheory for all results, unless explicitly stated otherwise.
\item[(II)] We adopt Kleene's notion of \emph{higher-order computation} as given by his nine clauses S1-S9 (see \cite{longmann}*{Ch.\ 5} or \cite{kleeneS1S9}) as our official notion of `computable'.
\end{enumerate}
We mention that S1-S8 are rather basic and merely introduce a kind of higher-order primitive recursion with higher-order parameters. 
The real power comes from S9, which essentially hard-codes the \emph{recursion theorem} for S1-S9-computability in an ad hoc way.  
By contrast, the recursion theorem for Turing machines is derived from first principles in \cite{zweer}.

\smallskip

On a historical note, it is part of the folklore of computability theory that many have tried (and failed) to formulate models of computation for objects of all finite types in which one derives the recursion theorem in a natural way.  For this reason, Kleene ultimately introduced S1-S9, which 
were initially criticised for their ad hoc nature, but eventually received general acceptance nonetheless.  
On a related note, an equivalent-but-more-elegant $\lambda$-calculus formulation of S1-S9-computability based on fixed point operators may be found in \cite{dagsamXIII}; an additional advantage 
of the framework from \cite{dagsamXIII} is that it accommodates partial objects.

\smallskip

We refer to \cite{longmann, dagsamXIII} for a (more) thorough overview of higher-order computability theory.
We do mention the distinction between `normal' and `non-normal' functionals  based on the following definition from \cite{longmann}*{\S5.4}. 
We only make use of $\exists^{n}$ for $n=2,3$, as defined in Section \ref{lll}.
\bdefi\label{norma}
For $n\geq 2$, a functional of type $n$ is called \emph{normal} if it computes Kleene's $\exists^{n}$ following S1-S9, and \emph{non-normal} otherwise.  
\edefi
\noindent
It is a historical fact that higher-order computability theory, based on Kleene's S1-S9 schemes, has focused primarily on the world of \emph{normal} functionals; this opinion can be found \cite{longmann}*{\S5.4}.  
Nonetheless, we have previously studied the computational properties of new \emph{non-normal} functionals, namely those that compute the objects claimed to exist by:
\begin{itemize}
\item covering theorems due to Heine-Borel, Vitali, and Lindel\"of (\cites{dagsam, dagsamII, dagsamVI}),
\item the Baire category theorem (\cite{dagsamVII}),
\item local-global principles like \emph{Pincherle's theorem} (\cite{dagsamV}),
\item weak fragments of the Axiom of (countable) Choice (\cite{dagsamIX}),
\item the uncountability of $\R$ and the Bolzano-Weierstrass theorem for countable sets in Cantor space (\cites{dagsamX, dagsamXI}),
\item the Jordan decomposition theorem and related results (\cite{dagsamXII, dagsamXIII}).
\end{itemize}
In this paper, we greatly extend the study of the uncountability of $\R$ mentioned in the fifth item; the operations sketched in Section \ref{intro} are all non-normal, in that they do not even compute Kleene's $\exists^{2}$ from Section \ref{lll}.

\subsubsection{Some comprehension functionals}\label{lll}
In Turing-style computability theory, computational hardness is measured in terms of where the oracle set fits in the well-known comprehension hierarchy.  
For this reason, we introduce some axioms and functionals related to \emph{higher-order comprehension} in this section.
We are mostly dealing with \emph{conventional} comprehension here, i.e.\ only parameters over $\N$ and $\N^{\N}$ are allowed in formula classes like $\Pi_{k}^{1}$ and $\Sigma_{k}^{1}$.  

\smallskip

First of all, the functional $\varphi^{2}$, also called \emph{Kleene's quantifier $\exists^{2}$}, as in $(\exists^{2})$ is clearly discontinuous at $f=11\dots$; in fact, $\exists^{2}$ is (computationally) equivalent to the existence of $F:\R\di\R$ such that $F(x)=1$ if $x>_{\R}0$, and $0$ otherwise via Grilliot's trick (see \cite{kohlenbach2}*{\S3}).
\be\label{muk}\tag{$\exists^{2}$}
(\exists \varphi^{2}\leq_{2}1)(\forall f^{1})\big[(\exists n)(f(n)=0) \asa \varphi(f)=0    \big]. 
\ee
Related to $(\exists^{2})$, the functional $\mu^{2}$ in $(\mu^{2})$ is called \emph{Feferman's $\mu$} (\cite{avi2}).
\begin{align}\label{mu}\tag{$\mu^{2}$}
(\exists \mu^{2})(\forall f^{1})\big(\big[ (\exists n)(f(n)=0) \di [f(\mu(f))=0&\wedge (\forall i<\mu(f))(f(i)\ne 0) \big]\\
& \wedge [ (\forall n)(f(n)\ne0)\di   \mu(f)=0] \big). \notag
\end{align}
We have $(\exists^{2})\asa (\mu^{2})$ over Kohlenbach's base theory (\cite{kohlenbach2}), while $\exists^{2}$ and $\mu^{2}$ are also computationally equivalent.  

\smallskip

Secondly, the functional $\SS^{2}$ in $(\SS^{2})$ is called \emph{the Suslin functional} (\cite{kohlenbach2}).
\be\tag{$\SS^{2}$}
(\exists\SS^{2}\leq_{2}1)(\forall f^{1})\big[  (\exists g^{1})(\forall n^{0})(f(\overline{g}n)=0)\asa \SS(f)=0  \big].
\ee
By definition, the Suslin functional $\SS^{2}$ can decide whether a $\Sigma_{1}^{1}$-formula as in the left-hand side of $(\SS^{2})$ is true or false.   
We similarly define the functional $\SS_{k}^{2}$ which decides the truth or falsity of $\Sigma_{k}^{1}$-formulas from $\L_{2}$.
%
We note that the operators $\nu_{n}$ from \cite{boekskeopendoen}*{p.\ 129} are essentially $\SS_{n}^{2}$ strengthened to return a witness (if existant) to the $\Sigma_{n}^{1}$-formula at hand.  

\smallskip

\noindent
Thirdly, the functional $E^{3}$ clearly computes $\exists^{2}$ and $\SS_{k}^{2}$ for any $k\in \N$:
\be\tag{$\exists^{3}$}
(\exists E^{3}\leq_{3}1)(\forall Y^{2})\big[  (\exists f^{1})(Y(f)=0)\asa E(Y)=0  \big].
\ee
The functional from $(\exists^{3})$ is also called \emph{Kleene's quantifier $\exists^{3}$}, and we use the same -by now obvious- convention for other functionals.  

\smallskip

In conclusion, the operations sketched in Section \ref{intro} are computable in $\exists^{3}$ but not in any $\SS_{k}^{2}$, as noted in Footnote \ref{klank} and which immediately follows from \cite{dagsamXII}*{Theorem 18}.  
Many non-normal functionals exhibit the same `computational hardness' and we merely view this as support for the development of a separate scale for classifying non-normal functionals.    

\subsubsection{Some definitions}\label{cdef}
We introduce some definitions needed in the below, mostly stemming from mainstream mathematics.

\smallskip

First of all, we shall study the following notions of weak continuity, going back (at least) to the days of Baire and Volterra.  
\bdefi\label{flung} For $f:[0,1]\di \R$, we have the following definitions:
\begin{itemize}
\item $f$ is \emph{upper semi-continuous} at $x_{0}\in [0,1]$ if $f(x_{0})\geq_{\R}\lim\sup_{x\di x_{0}} f(x)$,
\item $f$ is \emph{lower semi-continuous} at $x_{0}\in [0,1]$ if $f(x_{0})\leq_{\R}\lim\inf_{x\di x_{0}} f(x)$,
\item $f$ is \emph{quasi-continuous} at $x_{0}\in [0, 1]$ if for $ \epsilon > 0$ and an open neighbourhood $U$ of $x_{0}$, 
there is a non-empty open ${ G\subset U}$ with $(\forall x\in G) (|f(x_{0})-f(x)|<\eps)$.
\item $f$ is \emph{regulated} if for every $x_{0}$ in the domain, the `left' and `right' limit $f(x_{0}-)=\lim_{x\di x_{0}-}f(x)$ and $f(x_{0}+)=\lim_{x\di x_{0}+}f(x)$ exist.  
\end{itemize}
\edefi
Scheeffer studies discontinuous regulated functions in \cite{scheeffer} arround 1884 (without using the term `regulated'), while Bourbaki develops Riemann integration based on regulated functions in \cite{boerbakies}.  
An interesting observation about regular functions and continuity is as follows.
\begin{rem}[Continuity and regulatedness]\label{atleast}\rm
First of all, as discussed in \cite{kohlenbach2}*{\S3}, the \emph{local} equivalence for functions on Baire space between sequential and `epsilon-delta' continuity cannot be proved in $\ZF$.  
By \cite{dagsamXI}*{Theorem 3.32}, this equivalence for \emph{regulated} functions is provable in $\ZF$ (and much weaker systems).  

\smallskip

Secondly, $\mu^{2}$ readily computes the left and right limits of regulated $f:[0,1]\di \R$.  In this way, the formula `$f$ is continuous at $x\in [0,1]$' is decidable using $\mu^{2}$, namely equivalent to the formula `$f(x+)=f(x)=f(x-)$'.  
The usual `epsilon-delta' definition of continuity involves quantifiers over $\R$, i.e.\ the previous equality is much simpler and more elementary. 
\end{rem}
Secondly, we also need the notion of `intermediate value property', also called the `Darboux property'.  
\bdefi[Darboux property] Let $f:[0,1]\di \R$ be given. 
\begin{itemize}
\item A real $y\in \R$ is a left (resp.\ right) \emph{cluster value} of $f$ at $x\in [0,1]$ if there is $(x_{n})_{n\in \N}$ such that $y=\lim_{n\di \infty} f(x_{n})$ and $x=\lim_{n\di \infty}x_{n}$ and $(\forall n\in \N)(x_{n}\leq x)$ (resp.\ $(\forall n\in \N)(x_{n}\geq x)$).  
\item A point $x\in [0,1]$ is a \emph{Darboux point} of $f:[0,1]\di \R$ if for any $\delta>0$ and any left (resp.\ right) cluster value $y$ of $f$ at $x$ and $z\in \R$ strictly between $y$ and $f(x)$, there is $w\in (x-\delta, x)$ (resp.\ $w\in ( x, x+\delta)$) such that $f(w)=y$.   
\end{itemize}
\edefi
By definition, a point of continuity is also a Darboux point, but not vice versa.

\smallskip

Thirdly, as suggested in Section \ref{intro}, the set-theoretic definition of countable set is not suitable for the study of regulated functions (in computability theory).  
We present our alternative notion in Definition~\ref{hoogzalieleven}, which amounts to `unions over $\N$ of finite sets'.
We first provide some motivation in Remark~\ref{diunk}.
We assume subsets of $\R$ are given by their characteristic functions.  
\begin{rem}[Countable sets by any other name]\label{diunk}\rm
First of all, we have previously investigated the computational properties of the Jordan decomposition theorem in \cites{dagsamXII, dagsamXIII}.  
In the latter, a central object of study is as follows:
\begin{center}
\emph{a functional $\mathcal{E}$ that on input $A\subset [0,1]$ \textbf{and} $Y:[0,1]\di \N$ such that $Y$ is injective \(or bijective\) on $A$, outputs a sequence of reals listing all reals in $A$}.
\end{center}
We stress that the functional $\mathcal{E}$ crucially has \textbf{two} inputs: a countable set $A\subset [0,1]$ \textbf{and} a function $Y:[0,1]\di \N$ witnessing that $A$ is countable.  
Indeed, omitting the second input from $\mathcal{E}$ yields a choice function well-known from the Axiom of Choice.  
However, such functions lack any and all computational content (even in very weak instances: see \cite{dagsamIX}).  
Thus, the functional $\mathcal{E}$ is not a meaningful object of study in computability theory \emph{unless} we include the second input $Y$.  

\smallskip

%
\noindent
Secondly, consider the following sets definable via $\exists^{2}$, for regulated $f:[0,1]\di \R$:
\begin{eqnarray}\label{lagel2}\textstyle
&A:= \big\{x\in (0,1):  f(x+)\ne f(x) \vee f(x-)\ne f(x)\big\}, \notag \\
&A_{n}:=\big\{x\in (0,1): |f(x+)- f(x)|>\frac1{2^{n}} \vee |f(x-)- f(x)|>\frac1{2^{n}}\big\}.
\end{eqnarray}
Clearly, $A=\cup_{n\in\N}A_{n}$ collects all points in $(0,1)$ where $f$ is discontinuous; this set is central to many proofs involving regulated functions (see e.g.\ \cite{voordedorst}*{Thm.\ 0.36}).  
Now, that $A_{n}$ is finite follows by a standard\footnote{If $A_{n}$ were infinite, the Bolzano-Weierstrass theorem implies the existence of a limit point $y\in [0,1]$ for $A_{n}$.  One readily shows that $f(y+)$ or $f(y-)$ does not exist, a contradiction as $f$ is assumed to be regulated.\label{fkluk}} compactness argument.  
However, while $A$ is then countable, we are unable to construct an injection from $A$ to $\N$ (let alone a bijection), even assuming $\SS_{m}^{2}$.
Similarly, for a set $A\subset \R$ without limit points, $A\cap [-n,n]$ is finite for any $n\in \N$, yet we are again unable to construct an injection from $A$ to $\N$, even assuming $\SS_{m}^{2}$.

\smallskip

In conclusion, the countable set $A$ in \eqref{lagel2} is central to many proofs involving regulated functions, but it seems we cannot (in general) construct an injection from $A$ to $\N$.  
However, this injection is a crucial input of the functional $\mathcal{E}$ as omitting it gives rise to a choice function from the Axiom of Choice.  
%
%
\end{rem}
Besides pointing out a problem with the usual definition of countable set, Remark~\ref{diunk} also suggests an alternative, more suitable, definition of countable set, namely `union over $\N$ of finite sets'.  
We now introduce the `usual' definitions of countable set (Definitions \ref{eni} and \ref{standard}), and our alternative one (Definition \ref{hoogzalieleven}).  
\bdefi[Enumerable sets of reals]\label{eni}
A set $A\subset \R$ is \emph{enumerable} if there exists a sequence $(x_{n})_{n\in \N}$ such that $(\forall x\in \R)(x\in A\di (\exists n\in \N)(x=_{\R}x_{n}))$.  
\edefi
This definition reflects the RM-notion of `countable set' from \cite{simpson2}*{V.4.2}.  
We note that given $\mu^{2}$ from Section \ref{lll}, we may replace the final implication in Definition \ref{eni} by an equivalence. 
One sometimes finds `denumerable' in the literature to describe Definition \ref{eni}.
\bdefi[Countable subset of $\R$]\label{standard}~
A set $A\subset \R$ is \emph{countable} if there exists $Y:\R\di \N$ such that $(\forall x, y\in A)(Y(x)=_{0}Y(y)\di x=_{\R}y)$. 
If $Y:\R\di \N$ is also \emph{surjective}, i.e.\ $(\forall n\in \N)(\exists x\in A)(Y(x)=n)$, we call $A$ \emph{strongly countable}.
\edefi
The first part of Definition \ref{standard} is from Kunen's set theory textbook (\cite{kunen}*{p.~63}) and the second part is taken from Hrbacek-Jech's set theory textbook \cite{hrbacekjech} (where the term `countable' is used instead of `strongly countable').  
For the rest of this paper, `strongly countable' and `countable' shall exclusively refer to Definition~\ref{standard}, \emph{except when explicitly stated otherwise}. 
\begin{defi}\label{hoogzalieleven}
A set $A\subset \R$ is \emph{weakly countable} if there is a \emph{height} $H:\R\di \N$ for $A$, i.e.\ for all $n\in \N$, $A_{n}:= \{ x\in A: H(x)<n\}$ is finite.  
\end{defi}
We note that the notion of `height' is mentioned in e.g.\ \cite{demol}*{p.\ 33} and \cite{vadsiger}, from whence we took this notion.   The following remark is crucial for the below.
\begin{rem}[Inputs and heights]\label{inhe}\rm
Similar to the functional $\mathcal{E}$ in Remark \ref{diunk}, a functional defined on weakly countable sets always takes \textbf{two} inputs: the set $A\subset \R$ \textbf{and} the height $H:\R\di \N$.  
We note that given a sequence of finite sets $(X_{n})_{n\in \N}$ in $\R$, a height can be defined as $H(x):= (\mu n)(x\in X_{n})$ using Feferman's $\mu^{2}$. 
A set is therefore weakly countable iff it is the union over $\N$ of finite sets.  In particular, the set $A=\cup_{n\in \N}A_{n}$ as defined in \eqref{lagel2} has a height $H_{A}(x):= (\mu n)(x\in A_{n})$, 
which is crucial for the results in Section \ref{birf}.  
\end{rem}

\section{Main results}\label{main}
We establish the results sketched in Section \ref{intro}.
We generally assume $\exists^{2}$ from Section \ref{prelim} to avoid the technical details involved in the representation of sets and real numbers.  
Given the hardness of the operations from Section \ref{intro}, as noted in Footnote \ref{klank} and Section \ref{lll}, this seems like a weak assumption.

\subsection{Introduction}
The uncountability of $\R$ can be studied in numerous guises in higher-order computability theory.  
For instance, the following notions are from \cite{dagsamX, dagsamXII}, where it is also shown that many extremely basic operations compute these realisers.  
\bdefi[Realisers for the uncountability of $\R$]\label{kefi}~
\begin{itemize}
\item A \emph{Cantor functional/realiser} takes as input $A\subset [0,1]$ and $Y:[0,1]\di \N$ such that $Y$ is injective on $A$, and outputs $x\not \in A$.  
\item A \emph{\textbf{weak} Cantor realiser} takes as input $A\subset [0,1]$ and $Y:[0,1]\di \N$ such that $Y$ is \textbf{bijective} on $A$, and outputs $x\not \in A$.  
\item A $\NIN$-\emph{realiser} takes as input $Y:[0,1]\di \N$ and outputs $x,y\in [0,1]$ with $x\ne y \wedge Y(x)=Y(y)$.  
\end{itemize}
\edefi
Clearly, the realisers in Definition \ref{kefi} are based on the `set theoretic' definition of countability.  As discussed in Remark \ref{diunk}, the notion of weak countability (Definition \ref{hoogzalieleven}) seems to be have better, suggesting 
the following stronger notion. 
\begin{defi}[Strong Cantor realiser]
A \emph{strong Cantor realiser} is any functional that on input a weakly countable $A\subset [0,1]$, outputs $y\in [0,1]\setminus A$. 
\end{defi}
Following Remark \ref{inhe}, a strong Cantor realiser takes as input a countale set \textbf{and} its height as in Definition \ref{hoogzalieleven}.
Modulo $\exists^{2}$, this amounts to an input consisting of a sequence $(X_{n})_{n\in \N}$ of finite sets in $[0,1]$ and an output $y\in [0,1]\setminus \cup_{n\in \N}X_{n}$.
The reader will observe that a strong Cantor realiser computes a normal one.  Any (weak, strong, or normal) Cantor realiser witnesses the extremely basic idea that removing a countable set from an uncountable one, there is at least one element left.  
Nonetheless, the functionals $\SS_{k}^{2}$ from Section \ref{lll}, which decide $\Pi_{k}^{1}$-formulas, cannot compute a weak Cantor realiser (\cite{dagsamXII}*{Theorem 18}).

\smallskip

We stress that (strong, weak, or normal) Cantor realisers are \emph{partial} objects: they provide the required output $y\in [0,1]\setminus A$ when the inputs are given as in their respective specification.  
If the inputs are not according to the specification, the output may be undefined or may be some real number.   As noted in Section \ref{prelim}, there is an equivalent formulation of Kleene's S1-S9 schemes that accommodates partial objects. 

\smallskip

We show in Section \ref{birf} that many operations on regulated functions are computationally equivalent to strong Cantor realisers.  
This provides additional justification for the introduction of our new notion of weak countability. 

\subsection{Computational equivalences for strong Cantor realisers}\label{birf}
In this section, we show that many operations on regulated functions are computationally equivalent to strong Cantor realisers.  
We recall that subsets of $\R$ are given by characteristic functions.    

\smallskip

First of all, various natural functionals compute a strong Cantor realiser.
\begin{thm}\label{kifkif}
Given $\exists^{2}$, these functionals compute a strong Cantor realiser:
\begin{itemize}
\item the structure functional $\Omega$ \(see \cite{dagsamXIII}\) which on input finite $X\subset \R$ outputs a finite sequence $\Omega(X)=(x_{0}, \dots, x_{k})$ which includes all elements of $X$.
\item the non-monotone induction functional \(see \cite{dagsamVII, dagcomp20}\),
\item a realiser for the Baire category theorem \(see \cite{dagsamVII}*{\S6}\).
\end{itemize}
\end{thm}
\begin{proof}
Let $(X_{n})_{n\in \N}$ be a sequence of finite sets in $[0,1]$.  
Since $\Omega$ can enumerate finite sets, the sequence $\Omega(X_{0})*\Omega(X_{1})*\dots$ enumerates $\cup_{n\in \N}X_{n}$.  
One readily computes $y\in [0,1]$ not in this sequence (see e.g.\ \cite{simpson2}*{II.4.9} or \cite{grayk}), yielding a strong Cantor realiser.  

\smallskip

The second item computes the third item by \cite{dagsamVII}*{Theorem 6.5}.  Regarding the third item, $O_{n}:= [0,1]\setminus \cup_{k\leq n}X_{k}$ is (trivially) dense and open.  
Clearly any $y \in \cap_{n\in \N}O_{n}$ by definition satisfies $y\in [0,1]\setminus \cup_{n\in \N}X_{n}$, i.e.\ a realiser for the Baire category theorem computes a strong Cantor realiser. 
\qed
\end{proof}
%

Secondly, we have the following theorem. 
Regarding item \ref{full6}, the property `$x$ is a local strict$^{\ref{balia}}$ maximum of $f$' does not seem to be decidable given $\exists^{2}$ for (general) regulated $f$, in contrast to Remark \ref{atleast}.
We also note that items \ref{full2} and \ref{full25} are based on Volterra's Theorem \ref{VOL} and Corollary \ref{VOLcor}, while item \ref{full26} is based on the associated generalisation as in Theorem \ref{dorki}, and sports a certain robustness. 
\begin{thm}\label{flunk}
The following are computationally equivalent modulo $\exists^{2}$.
\begin{enumerate}
\renewcommand{\theenumi}{\(\alph{enumi}\)}
\item A strong Cantor realiser.\label{full3}
\item A functional that on input regulated $f:[0,1]\di \R$, outputs $a, b\in  [0,1]$ such that $\{ x\in [0,1]:f(a)\leq f(x)\leq f(b)\}$ is infinite.\label{full0}
\item A functional that on input regulated $f:[0,1]\di \R$, outputs $y\in  [0,1]$ where $f$ is continuous.\label{full1}
\item A functional that on input regulated $f:[0,1]\di \R$, outputs $y\in  [0,1]$ where $f$ is {quasicontinuous}, or lower semi-continuous, or Darboux.\label{full12}
\item A `density' functional that on input regulated $f:[0,1]\di \R$, $x\in [0,1]$, and $k\in \N$, outputs $y\in  [0,1]\cap B(x, \frac{1}{2^{k}})$ such that $f$ is continuous at $y$. \label{full1.5}
\item A functional that on input regulated $f:[0,1]\di \R$, outputs either $q\in \Q\cap [0,1]$ where $f$ is discontinuous, or $x\in [0,1]\setminus \Q$ where $f$ is continuous. \label{full2}
\item A functional that on input regulated $f,g:[0,1]\di \R$, outputs a real $x\in [0,1]$ such that $f$ and $g$ are both continuous or both discontinuous at $x$. \label{full25}
\item A functional that on input regulated $f,g:[0,1]\di \R$ and weakly countable \(or countable, or strongly countable\) and dense $D\subset [0,1]$, outputs either $d\in D\cap [0,1]$ where $f$ is discontinuous, or $x\in [0,1]\setminus D$ where $f$ is continuous.\label{full26}
\item A functional that on input regulated $f:[0,1]\di [0,1]$ with Riemann integral $\int_{0}^{1}f(x)dx=0$, outputs $x\in [0,1]$ with $f(x)=0$ \(Bourbaki, \cite{boereng}*{p.\ 61, Cor.~1}\).\label{full4}
\item A functional that on input regulated $f:[0,1]\di \R$, outputs $y\in (0,1)$ where $F(x):=\lambda x.\int_{0}^{x}f(t)dt$ is differentiable with derivative equal to $f(y)$.\label{full5}
\item A functional that on input regulated $f:[0,1]\di \R$ with only removable discontinuities, outputs $x\in [0,1]$ which is \textbf{not} a strict\footnote{A point $x\in [0,1]$ is a strict local maximum of $f:[0,1]\di \R$ in case $(\exists N\in \N)( \forall y \in B(x, \frac{1}{2^{N}}))(y\ne x\di f(y)<f(x))$.\label{balia}} local maximum. \label{full6}
\end{enumerate}
\end{thm}
\begin{proof}
First of all, assume item \ref{full0} and let $X:=\cup_{n\in \N}X_{n}$ be the union of finite sets $X_{n}\subset [0,1]$ and define $h$ as follows:
\be\label{modi3}
h(x):=
\begin{cases}
0 & x\not \in X \\
\frac{1}{2^{n+1}} &  x\in X_{n} \textup{ and $n$ is the least such number}
\end{cases}.
\ee
Since $\cup_{k\leq n}X_{k}$ is finite for all $n\in \N$,  we observe that $h$ is regulated, in particular $0=h(0+)=h(1-)=h(x+)=h(x-)$ for any $x\in (0,1)$.
Now, in case $h(a)>0$, then $\{ x\in [0,1]:h(a)\leq h(x)\leq h(b)\}$ is finite by the definition of $h$.  In case $h(a)=0$, \eqref{modi3} implies $a\in [0,1]\setminus \cup_{n\in \N}X_{n}$, and item \ref{full3} follows.  

\smallskip

Secondly, assume item \ref{full1} and note that if $f:[0,1]\di \R$ is continuous at $y\in [0,1]$, we can use $\mu^{2}$ to find $N\in \N$ such that $(\forall q\in B(y, \frac{1}{2^{N}})\cap \Q\cap [0,1]  )(|f(y)-f(q)|<\frac{1}{2}) $; this readily 
yields $a, b\in [0,1]$ as required by item \ref{full0}.  We note that item \ref{full1.5} immediately yields item \ref{full1}, while the reversal readily follows by rescaling. 
Moreover, item \ref{full1} implies item \ref{full12}; 
applying the latter to $h$ as in \eqref{modi3}, we observe that any point $y\in [0,1]$ where $h$ is quasi- or lower semi-continuous or Darboux, is such that $y\in [0,1]\setminus \cup_{n\in \N}X_{n}$, i.e.\ item \ref{full3} follows.

\smallskip

Thirdly, assume item \ref{full2} and fix regulated $f:[0,1]\di \R$.  The following case distinction is decidable using $\exists^{2}$:
\begin{itemize}
\item there is $q\in \Q\cap [0,1]$ with $f(q+)=f(q)= f(q-)$, or
\item for all $q\in \Q\cap [0,1]$, we have $f(q+)\ne f(q) \vee f(q-)\ne f(q)$
\end{itemize}
In the first case, use $\mu^{2}$ to find this rational.  In the second case, the output of item \ref{full2} must be $x\in [0,1]\setminus \Q$ such that $f$ is continuous at $x$. 
In each case, we have a point of continuity for $f$, i.e.\ item \ref{full1} follows. 

\smallskip

\smallskip

Fourth, assume item \ref{full3} and fix regulated $f:[0,1]\di \R$.   Note that $\mu^{2}$ can find $q\in [0,1]\cap \Q$ with $f(q+)\ne f(q) \vee f(q-)\ne f(q)$, if such rational exists.   
In case $f$ is continuous at all $q\in [0,1]\cap \Q$, let $(q_{m})_{m\in \N}$ be an enumeration of $\Q\cap [0,1]$. 
Now consider the following set 
\be\label{sameold}\textstyle
X_{n}:=\big\{x\in (0,1): |f(x+)- f(x)|>\frac1{2^{n}} \vee |f(x-)- f(x)|>\frac1{2^{n}}\big\}, 
\ee
which is finite by Footnote \ref{fkluk}.  
Now use the strong Cantor realiser to find $y\in [0,1]\setminus \big(   \cup_{n\in \N} Y_{n}  \big)$, where $Y_{n}=X_{n}\cup \{q_{n}\}$.
Then $f$ must be continuous at $y \in [0,1]\setminus \Q$ and item \ref{full2} follows. 

\smallskip

Fifth, assume item \ref{full4} and let $h$ be as in \eqref{modi3}. 
Since $h$ is regulated, it is Riemann integrable, with $\int_{0}^{1}h(x)dx=0$.  Any $y\in [0,1]$ such that $h(y)=0$ is by definition not in $\cup_{n\in \N}X_{n}$, yielding a strong Cantor realiser as in item \ref{full3}. 
Item \ref{full1} readily yields item \ref{full4} as follows: \emph{if} regulated $f:[0,1]\di [0,1]$ has $\int_{0}^{1}f(x)dx=0$ and is continuous at $y\in [0,1]$, \emph{then} we must have $f(y)=0$;
indeed, in case $f(y)>0$, the continuity of $f$ at $y$ implies that $f(z)>\frac{1}{2^{k}}$ for $z\in B(y,\frac{1}{2^{k}} )$ for some $k\in \N$, implying $\int_{0}^{1}f(x)dx > 0$, since $f$ is non-negative on $[0,1]$.

\smallskip

Sixth, item \ref{full1} implies item \ref{full5} by the (second) fundamental theorem of calculus (see e.g.\ \cite{nudyrudy}*{Thm.\ 8.17, p.\ 165}). 
Item \ref{full5} applied to $h$ as in \eqref{modi3} yields item \ref{full3} as $\int_{0}^{x}h(t)dt =0$ for all $x\in [0,1]$.  Indeed, $h(x_{0})=\int_{0}^{x_{0}}h(t)dt =0$ implies $x_{0}\in [0,1]\setminus \cup_{n\in \N}X_{n}$ as required for a strong Cantor realiser. 
  
\smallskip

Seventh, consider item \ref{full6} and note that $h$ as in \eqref{modi3} is regulated with only removable discontinuities.  
Now, the set $X=\cup_{n\in \N}X_{n}$ consists of the local strict maxima of $h$, i.e.\ item \ref{full6} yields a strong Cantor realiser. 
For the reversal, $\exists^{2}$ computes a functional $M$ such that $M(g, a, b)$ is a maximum of $g\in C([0,1])$ on $[a,b]\subset [0,1]$ (see \cite{kohlenbach2}*{\S3}), i.e.\ $(\forall y\in [a,b])(g(y)\leq g(M(g, a,b)))$.
Using the functional $M$, one readily shows that `$x$ is a strict local maximum of $g$' is decidable\footnote{If $g\in C([0,1])$, then $x\in [0,1]$ is a strict local maximum iff for some $\epsilon\in \Q^{+}$:
\begin{itemize}
\item $g(y) < g(x)$ whenever $|x-y] < \epsilon$ for any $q\in [0,1]\cap\Q $, and: 
\item $\sup_{y\in [a,b]}g(y) < g(x)$ whenever $x \not \in [a,b]$, $a,b\in \Q$ and $[a,b] \subset [x - \epsilon,x + \epsilon]$.
\end{itemize}
Note that $\mu^{2}$ readily yields $N\in \N$ such that $(\forall y\in B(x, \frac{1}{2^{N}}))( g(y)<g(x))$.\label{roofer}
} given $\exists^{2}$, for $g$ continuous on $[0,1]$.  
Now let $f:[0,1]\di \R$ be regulated and with only removable discontinuities.  
Use $\exists^{2}$ to define $\tilde{f}:[0,1]\di \R$ as follows: $\tilde{f}(x):= f(x+)$ for $x\in [0, 1)$ and $\tilde{f}(1)=f(1-)$.
By definition, $\tilde{f}$ is continuous on $[0,1]$, and $\exists^{2}$ computes a (continuous) modulus of continuity, which follows in the same way as for Baire space (see e.g.\ \cite{kohlenbach4}*{\S4}).
In case $f$ is discontinuous at $x\in [0,1]$, the latter point is a strict local maximum of $f$ if and only if $f(x)>f(x+)$ (or $f(x)>f(x-)$ in case $x=1$).  
Note that $\mu^{2}$ (together with a modulus of continuity for $\tilde{f}$) readily yields $N_{f, x}\in \N$ such that $(\forall y\in B(x, \frac{1}{2^{N_{f,x}}}))( f(y)<f(x))$, in case $x$ is a strict local maximum of $f$.
In case $f$ is continuous at $x\in [0,1]$, we can use $\exists^{2}$ to decide whether $x$ is a local strict maximum of $\tilde{f}$.  
By Footnote \ref{roofer}, $\mu^{2}$ again yields $N_{f, x}\in \N$ such that $(\forall y\in B(x, \frac{1}{2^{N}}))( \tilde{f}(y)<\tilde{f}(x)))$, in case $x$ is a strict local maximum of $\tilde{f}$.
Now consider the following set:
\[\textstyle
A_{n}:=\{x\in [0,1]:  \textup{$x$ is a strict local maximum of $\tilde{f}$ or $f$ with $n\geq N_{{f}, x}$}\}.
\]
Then $A_{n}$ has finitely elements as strict local maxima cannot be `too close'.  
Hence, a strong Cantor realiser yields $y \in [0,1]\setminus \cup_{n\in \N}A_{n}$, which is not a local maximum of $f$, i.e.\ item \ref{full6} follows. 

\smallskip

Next, item \ref{full26} implies \ref{full2} as $D=\Q$ is trivially (weakly or strongly) countable.
To derive item \ref{full26} from item \ref{full3}, consider $X_{n}$ as in \eqref{sameold} and assume $D=\cup_{n\in \N}D_{n}$ is dense, where $D_{n}$ is finite for all $n\in \N$. 
Note that $\mu^{2}$ can find $q\in [0,1]\cap D\cap \Q$ with $f(q+)\ne f(q) \vee f(q-)\ne f(q)$, if such rational exists.   
In case $f$ is continuous at all $q\in [0,1]\cap D\cap \Q$,  
 use a strong Cantor realiser to find $y\in [0,1]\setminus \big(   \cup_{n\in \N} Z_{n}  \big)$, where $Z_{n}=X_{n}\cup D_{n}$.  By definition, $y\in [0,1]\setminus D$ and $f$ is continuous at $y$.  

\smallskip

%
%
Finally, assume item \ref{full25} and note that Thomae's function from \eqref{thomae} is regulated as the left and right limits are $0$.  Hence, we obtain item \ref{full2}, as Thomae's function is continuous on $\R\setminus \Q$ and discontinuous on $\Q$.  To obtain item \ref{full25} from a strong Cantor realiser, let $f, g:[0,1]$ be regulated and consider the finite set $X_{n}$ as in \eqref{sameold}.
Let $Y_{n}$ be the same set for $g$. Then use a strong Cantor realiser to obtain $y\in [0,1]\setminus \big(\cup_{n\in \N}(X_{n}\cup Z_{n})\big)$.  By definition, $f$ and $g$ are continuous at $y$, i.e.\ item \ref{full25} follows.
%
\qed
\end{proof}
Regarding item \ref{full6}, one readily shows that the following functional is computationally equivalent to $\Omega$ from Theorem \ref{kifkif}.  
Note that the set of strict local maxima is countable for \emph{any} $\R\di \R$-function by \cite{saks}*{p.\ 261, Theorem (1.1)}.
\begin{itemize}
\item A functional that on input regulated $f:[0,1]\di \R$ with only removable discontinuities, outputs an enumeration of all strict local maxima. 
\end{itemize}
In conclusion, we have established a number of equivalences for strong Cantor realisers, while many more variations are possible.

%

\section{Some details of Kleene's computability theory}\label{appendisch}
Kleene's computability theory borrows heavily from type theory and higher-order arithmetic.  We briefly sketch some of the associated notions.  

\smallskip

First of all, Kleene's S1-S9 schemes define `$\Psi$ is computable in terms of $\Phi$' for objects $\Phi,\Psi$ of any finite type.  
Now, the collection of \emph{all finite types} $\mathbf{T}$ is defined by the following two clauses:
\begin{center}
(i) $0\in \mathbf{T}$   and   (ii)  If $\sigma, \tau\in \mathbf{T}$ then $( \sigma \di \tau) \in \mathbf{T}$,
\end{center}
where $0$ is the type of natural numbers, and $\sigma\di \tau$ is the type of mappings from objects of type $\sigma$ to objects of type $\tau$.
In this way, $1\equiv 0\di 0$ is the type of functions from numbers to numbers, and  $n+1\equiv n\di 0$.  
We view sets $X$ of type $\sigma$ objects as given by characteristic functions $F_{X}^{\sigma\di 0}$.  

\smallskip

Secondly, for variables $x^{\rho}, y^{\rho}, z^{\rho},\dots$ of any finite type $\rho\in \mathbf{T}$,   types may be omitted when they can be inferred from context.  
The constants include the type $0$ objects $0, 1$ and $ <_{0}, +_{0}, \times_{0},=_{0}$  which are intended to have their usual meaning as operations on $\N$.
Equality at higher types is defined in terms of `$=_{0}$' as follows: for any objects $x^{\tau}, y^{\tau}$, we have
\be\label{aparth}
[x=_{\tau}y] \equiv (\forall z_{1}^{\tau_{1}}\dots z_{k}^{\tau_{k}})[xz_{1}\dots z_{k}=_{0}yz_{1}\dots z_{k}],
\ee
if the type $\tau$ is composed as $\tau\equiv(\tau_{1}\di \dots\di \tau_{k}\di 0)$.  

\smallskip

Thirdly, we introduce the usual notations for common mathematical notions, like real numbers, as also can be found in \cite{kohlenbach2}.  
\begin{defi}[Real numbers and related notions]\label{keepintireal}\rm~
\begin{enumerate}
 \renewcommand{\theenumi}{\alph{enumi}}
\item Natural numbers correspond to type zero objects, and we use `$n^{0}$' and `$n\in \N$' interchangeably.  Rational numbers are defined as signed quotients of natural numbers, and `$q\in \Q$' and `$<_{\Q}$' have their usual meaning.    
\item Real numbers are coded by fast-converging Cauchy sequences $q_{(\cdot)}:\N\di \Q$, i.e.\  such that $(\forall n^{0}, i^{0})(|q_{n}-q_{n+i}|<_{\Q} \frac{1}{2^{n}})$.  
We use Kohlenbach's `hat function' from \cite{kohlenbach2}*{p.\ 289} to guarantee that every $q^{1}$ defines a real number.  
\item We write `$x\in \R$' to express that $x^{1}:=(q^{1}_{(\cdot)})$ represents a real as in the previous item and write $[x](k):=q_{k}$ for the $k$-th approximation of $x$.    
\item Two reals $x, y$ represented by $q_{(\cdot)}$ and $r_{(\cdot)}$ are \emph{equal}, denoted $x=_{\R}y$, if $(\forall n^{0})(|q_{n}-r_{n}|\leq {2^{-n+1}})$. Inequality `$<_{\R}$' is defined similarly.  
We sometimes omit the subscript `$\R$' if it is clear from context.           
\item Functions $F:\R\di \R$ are represented by $\Phi^{1\di 1}$ mapping equal reals to equal reals, i.e.\ extensionality as in $(\forall x , y\in \R)(x=_{\R}y\di \Phi(x)=_{\R}\Phi(y))$.\label{EXTEN}
\item The relation `$x\leq_{\tau}y$' is defined as in \eqref{aparth} but with `$\leq_{0}$' instead of `$=_{0}$'.  Binary sequences are denoted `$f^{1}, g^{1}\leq_{1}1$', but also `$f,g\in C$' or `$f, g\in 2^{\N}$'.  Elements of Baire space are given by $f^{1}, g^{1}$, but also denoted `$f, g\in \N^{\N}$'.
\item Sets of type $\rho$ objects $X^{\rho\di 0}, Y^{\rho\di 0}, \dots$ are given by their characteristic functions $F^{\rho\di 0}_{X}\leq_{\rho\di 0}1$, i.e.\ we write `$x\in X$' for $ F_{X}(x)=_{0}1$. \label{koer} 
\end{enumerate}
\end{defi}
For completeness, we list the following notational convention for finite sequences.  
\begin{nota}[Finite sequences]\label{skim}\rm
The type for `finite sequences of objects of type $\rho$' is denoted $\rho^{*}$, which we shall only use for $\rho=0,1$.  
We shall not always distinguish between $0$ and $0^{*}$. 
Similarly, we assume a fixed coding for finite sequences of type $1$ and shall make use of the type `$1^{*}$'.  
In general, we do not always distinguish between `$s^{\rho}$' and `$\langle s^{\rho}\rangle$', where the former is `the object $s$ of type $\rho$', and the latter is `the sequence of type $\rho^{*}$ with only element $s^{\rho}$'.  The empty sequence for the type $\rho^{*}$ is denoted by `$\langle \rangle_{\rho}$', usually with the typing omitted.  

\smallskip

Furthermore, we denote by `$|s|=n$' the length of the finite sequence $s^{\rho^{*}}=\langle s_{0}^{\rho},s_{1}^{\rho},\dots,s_{n-1}^{\rho}\rangle$, where $|\langle\rangle|=0$, i.e.\ the empty sequence has length zero.
For sequences $s^{\rho^{*}}, t^{\rho^{*}}$, we denote by `$s*t$' the concatenation of $s$ and $t$, i.e.\ $(s*t)(i)=s(i)$ for $i<|s|$ and $(s*t)(j)=t(|s|-j)$ for $|s|\leq j< |s|+|t|$. For a sequence $s^{\rho^{*}}$, we define $\overline{s}N:=\langle s(0), s(1), \dots,  s(N-1)\rangle $ for $N^{0}<|s|$.  
For a sequence $\alpha^{0\di \rho}$, we also write $\overline{\alpha}N=\langle \alpha(0), \alpha(1),\dots, \alpha(N-1)\rangle$ for \emph{any} $N^{0}$.  By way of shorthand, 
$(\forall q^{\rho}\in Q^{\rho^{*}})A(q)$ abbreviates $(\forall i^{0}<|Q|)A(Q(i))$, which is (equivalent to) quantifier-free if $A$ is.   
\end{nota}

\begin{ack}\rm
We thank Anil Nerode for his valuable advice.
Our research was supported by the \emph{Deutsche Forschungsgemeinschaft} via the DFG grant SA3418/1-1.
Initial results were obtained during the stimulating MFO workshop (ID 2046) on proof theory and constructive mathematics in Oberwolfach in early Nov.\ 2020.  
We express our gratitude towards the aforementioned institutions.    
\end{ack}

\section{Bibliography}

\bye